\documentclass[12pt]{amsart}
\usepackage{amscd,amssymb,amsthm,amsmath,amssymb,mathrsfs,enumerate,longtable}
\usepackage[matrix,arrow,curve]{xy}
\usepackage[margin=2cm]{geometry}
\usepackage{color}

\newtheorem*{theorem*}{Theorem}

\newtheorem{proposition}[equation]{Proposition}
\newtheorem*{proposition*}{Proposition}
\newtheorem{lemma}[equation]{Lemma}
\newtheorem{corollary}[equation]{Corollary}
\newtheorem*{corollary*}{Corollary}

\newtheorem*{problem*}{Problem}
\newtheorem*{lemma*}{Lemma}

\newtheorem*{question*}{Question}

\newtheorem*{construction*}{Construction}
\newtheorem*{maintheorem*}{Main Theorem}

\theoremstyle{definition}
\newtheorem{example}[equation]{Example}
\newtheorem*{example*}{Example}

\newtheorem*{definition*}{Definition}

\theoremstyle{remark}
\newtheorem{remark}[equation]{Remark}
\newtheorem*{remark*}{Remark}

\makeatletter\@addtoreset{equation}{section} \makeatother

\author[Abban, Cheltsov, Dubouloz, Fujita, Kishimoto, Park]{Hamid Abban, Ivan Cheltsov, Adrien Dubouloz, \\ Kento Fujita, Takashi Kishimoto, Jihun Park}

\title{On K-stability of singular hyperelliptic Fano 3-folds}

\address{\emph{Hamid Abban}\newline
\textnormal{University of Nottingham, Nottingham, England
\newline
\texttt{hamid.abban@nottingham.ac.uk}}}

\address{ \emph{Ivan Cheltsov}\newline
\textnormal{University of Edinburgh, Edinburgh, Scotland
\newline
\texttt{i.cheltsov@ed.ac.uk}}}

\address{ \emph{Adrien Dubouloz}\newline
\textnormal{University of Poitiers, Poitiers, France
\newline
\texttt{adrien.dubouloz@math.cnrs.fr}}}

\address{\emph{Kento Fujita}
\newline
\textnormal{Osaka University, Osaka, Japan}
\newline
\textnormal{\texttt{fujita@math.sci.osaka-u.ac.jp}}}

\address{ \emph{Takashi Kishimoto}\newline \textnormal{Saitama University, Saitama, Japan
\newline
\texttt{kisimoto.takasi@gmail.com}}}

\address{ \emph{Jihun Park}\newline \textnormal{Institute for Basic Science, Pohang, Korea \newline POSTECH, Pohang, Korea \newline
\texttt{wlog@postech.ac.kr}}}

\let\origmaketitle\maketitle
\def\maketitle{
  \begingroup
  \def\uppercasenonmath##1{} 
  \let\MakeUppercase\relax 
  \origmaketitle
  \endgroup
}

\begin{document}

\begin{abstract}
We study the K-stability of singular Fano 3-folds with canonical Gorenstein singularities whose anticanonical linear system is base-point-free but not very ample.
\end{abstract}

\maketitle

\section{Introduction}
\label{section:introduction}
Let $X$ be a Fano 3-fold with canonical Gorenstein singularities.
Then its (anticanonical) degree, defined as $(-K_X)^3$, is an even integer allowing one to introduce the genus of $X$ as
$$
g=\frac{(-K_X)^3}{2}+1.
$$
The linear system $|-K_X|$ has dimension $g+1$,
and it gives a rational map $\phi\colon X\dasharrow \mathbb{P}^{g+1}$, which in general is not a morphism. 
Suppose that $|-K_X|$ is free from base points (see Remark\,\ref{remark:base-points} below for the case when $|-K_X|$ has base points). If $-K_X$ is not very ample, then $X$ is called hyperelliptic as
$\phi$ induces a double cover $X\to\phi(X)\subseteq \mathbb{P}^{g+1}$, with further property that
\begin{itemize}
\item either $g=2$, $\phi(X)=\mathbb{P}^3$, and $\phi\colon X\to\mathbb{P}^3$ is branched over a sextic surface,
\item or $g\geqslant 3$ and $\phi(X)$ is a~3-fold in $\mathbb{P}^{g+1}$ of the degree $g-1$.
\end{itemize}

In this case, the 3-fold $X$ is contained in one of $47$ deformation families described in~\cite{ChPrSh}, following  which we will denote these Families as $H_1, H_2, H_3, H_4, \ldots, H_{46}, H_{47}$. The intention of this paper is to explore K-stability for members of these families, examined by some new techniques (\cite{Fujita2024}) as well as more routine tools.

Note that general members of Families $H_1$, $H_2$, $H_3$, $H_4$, $H_6$, $H_9$ are smooth, and all smooth members of these Families  are known to be K-stable \cite{Book,CheltsovParkWonMZ,Dervan2,CheltsovDenisovaFujita}.
On the other hand, all members of the remaining families are singular, and will be subject of study in this article. We will prove the following result.

\begin{maintheorem*}
The following assertions hold:
\begin{enumerate}
\item[$(\mathrm{1})$]
General members of Families $H_5,H_7,H_8,H_{11},H_{12},H_{13}$ are K-stable.

\item[$(\mathrm{2})$]
Family $H_{10}$ contains a K-polystable member, which is isomorphic to the~hypersurface
\begin{equation}
\label{equation:H10}
\{w^2=z^3x+t^3y\}\subset\mathbb{P}(1_x,1_y,3_z,3_t,5_w).
\end{equation}
Family $H_{17}$ contains a K-polystable member, which is isomorphic to the~hypersurface
\begin{equation}
\label{equation:H17}
\{w^2=tz(t+z)\}\subset\mathbb{P}(1_x,1_y,4_z,4_t,6_w).
\end{equation}

\item[$(\mathrm{3})$]
All members of Families $H_{14},H_{15}, H_{16}, H_{18},\dots, H_{47}$ are K-unstable.
\end{enumerate}
\end{maintheorem*}

In the body of this paper, we give a uniform proof of this theorem using \cite{Fujita2024}, however we will illustrate in Subsections \ref{alternative-unstable}, \ref{alternative-stable}, and \ref{alternative-polystable} other methods of tackling the problem using 
more traditional tools such as $\alpha$-invariant introduced by Tian \cite{CheltsovShramovUMN,Ji17}, Abban--Zhuang theory \cite{AbbanZhuang}, and K-polystability of complexity one $\mathbb{T}$-varieties \cite[\S~1.3]{Book}.

\begin{remark}
\label{remark:base-points}
If $|-K_X|$ has base points, then $X$ is one of the 3-folds described in \cite[Theorem~1.1]{JaRa04}.
In particular, if $X$ is smooth, then either $(-K_X)^3=6$ and $X=S\times\mathbb{P}^1$, where $S$ is a smooth del Pezzo surface of degree $1$, or $(-K_X)^3=4$ and $X$ is a blow up of a smooth hypersurface of degree $6$ in $\mathbb{P}(1^3,2,3)$ along a smooth complete intersection curve of genus $1$.

In both cases, the 3-fold $X$ is K-stable if it is smooth \cite{Book,CheltsovDenisovaFujita}.
On the other hand, if $(-K_X)^3\geqslant 8$, then it follows from \cite[Theorem~1.1]{JaRa04} or the proof of \cite[Theorem~2]{Cheltsov1999} that $\alpha(X)\leqslant\frac{1}{g}\leqslant\frac{1}{5}$, 
which implies that $X$ is K-unstable~\cite[Theorem 3.5]{FO}.
\end{remark}

\begin{remark}
\label{remark:H17}
Suppose that $X$ is contained in Family~$H_{17}$.
Then $X$ is a~hypersurface in the weighted projective space $\mathbb{P}(1_x,1_y,4_z,4_t,6_w)$, which can be given by the following equation:
$$
w^2=f_3(z,t)+a_4(x,y)z^2+b_4(x,y)zt+c_4(x,y)t^2+d_8(x,y)z+e_8(x,y)t+g_{12}(x,y),
$$
where $f_{3}(z,t)$ is a~non-zero cubic form,
$a_4(x,y)$, $b_4(x,y)$, $c_4(x,y)$ are forms of degree $4$,
both $d_8(x,y)$ and $e_8(x,y)$ are forms of degree $8$,
and $g_{12}(x,y)$ is a~form of degree $12$.
Now, changing coordinates, we may assume that one of the~following cases holds:
\begin{enumerate}
\item $f_3=z^3$, $a_4=0$;
\item $f_3=z^2t$, $a_4=0$, $b_4=0$;
\item $f_3=zt(z+t)$, $a_4=0$, $c_4=0$.
\end{enumerate}
Proving that $X$ is K-unstable in the first two cases is left to enthusiastic readers as exercise.
In the third case, scaling $x$ and $y$, we obtain an isotrivial degeneration of the 3-fold $X$ to \eqref{equation:H17},
which is K-polystable by Main Theorem.
Thus, it follows from \cite{BlumLiu,BlumXu} that either $X$ is isomorphic to~\eqref{equation:H17}, or $X$ is strictly K-semistable.
Therefore, \eqref{equation:H17} is the unique K-polystable member of this deformation family.
Similarly, we see that~\eqref{equation:H10} is the unique K-polystable member of Family~$H_{10}$.
\end{remark}

Using the Main Theorem, \cite{ChPrSh}, and Remark \ref{remark:base-points} above, we obtain the following corollary.

\begin{corollary}
\label{corollary:very-ample}
Let $X$ be a~Fano 3-fold with canonical Gorenstein singularities and $(-K_X)^3\geqslant 14$.
Suppose that $X$ is K-semistable. Then $-K_X$ is very ample.
\end{corollary}

Let us say a few words about the classification of hyperelliptic Fano 3-folds in \cite{ChPrSh}. Recall that
Family $H_1$ consists of hypersurfaces of degree $6$ in $\mathbb{P}(1^3,2,3)$,
Family $H_2$ consists of hypersurfaces of degree $6$ in~$\mathbb{P}(1^4,3)$,
and Family $H_{3}$ consists of complete intersections of a~quadric cone and a~quartic hypersurface in $\mathbb{P}(1^5,2)$.
Let $X$ be a~hyperelliptic Fano 3-fold in Family $H_n$ with $n\geqslant 4$.
Then it follows from \cite{ChPrSh} that there exist a~crepant birational morphism $\phi\colon V\to X$
and a~double cover $\pi\colon V\to\mathbb{F}(d_{1},d_{2},d_{3})$, where
$$
\mathbb{F}(d_{1},d_{2},d_{3})=\mathrm{Proj}\big(\mathcal{O}_{\mathbb{P}^{1}}(d_{1})\oplus\mathcal{O}_{\mathbb{P}^{1}}(d_{2})\oplus\mathcal{O}_{\mathbb{P}^{1}}(d_{3})\big)
$$
for some integers $d_1\geqslant d_2\geqslant d_3\geqslant 0$. The double cover $\pi$ is branched over a~(possibly reducible) reduced surface $S\subset\mathbb{F}(d_{1},d_{2},d_{3})$ such that
$$
S\sim 4M+2(2-d_1-d_2-d_3)L,
$$
where $M$ is the~tautological line bundle on the~scroll $\mathbb{F}(d_{1},d_{2},d_{3})$,
and $L$ is a~fiber of the~natural projection $\mathbb{F}(d_{1},d_{2},d_{3})\to\mathbb{P}^{1}$.
Let $\eta\colon V\to\mathbb{P}^1$ be the~composition of the double cover $\pi$ with the~natural projection $\mathbb{F}(d_1,d_2,d_3)\to\mathbb{P}^1$,
and let $F$ be a~sufficiently general fiber of the constructed morphism~$\eta$. Then $F$ is a~del Pezzo surface of degree $2$ that has at most Du Val singularities.
The paper \cite{ChPrSh} lists all possibilities for the triple $(d_1,d_2,d_3)$ such that the 3-fold $X$ with the required properties exists,
and \cite[Table~1]{ChPrSh} presents the defining equation of the surface $S$ using the coordinates $(x_1:x_2:x_3;t_1:t_2)$ on the~scroll $\mathbb{F}(d_{1},d_{2},d_{3})$ introduced in \cite{Re97},
and also describes the singularities of the del Pezzo surface $F$ under the~assumption that $S$ is a sufficiently general surface in the linear system $|4M+2(2-d_1-d_2-d_3)L|$.

\medskip
\noindent
\textbf{Acknowledgements.}
Hamid Abban is supported by Royal Society International Collaboration Award ICA$\backslash$1$\backslash$231019.
Ivan Cheltsov is supported by Simons Collaboration grant \emph{Moduli of Varieties} and JSPS Invitational Fellowships for Research in Japan (S24067).
Kento Fujita is supported by JSPS KAKENHI Grant Number 22K03269, 
Royal Society International Collaboration Award 
ICA\textbackslash 1\textbackslash 231019, and Asian Young Scientist Fellowship. 
Takashi Kishimoto is supported by JSPS KAKENHI Grant Number 23K03047. Jihun Park is supported by IBS-R003-D1, Institute for Basic Science in Korea. 
The authors express their gratitude to CIRM, Luminy, for its hospitality during one semester-long \emph{Morlet Chair} and for providing a perfect work environment.

\section{Proof of the Main Theorem}
\label{subsection:H-proof}

The self-contained proof of the Main Theorem uses  \cite[Corollary 7.7]{BlumJonsson} and results proved in \cite[\S 5 and \S 8]{Fujita2024}. 
Throughout, we will use the notation introduced in the end of Section~\ref{section:introduction}. 
Set $\mathbb{F}=\mathbb{F}(d_1,d_2,d_3)$ and let $S\subset\mathbb{F}$ be a surface with 
$$
S\sim 4M+2(2-d_1-d_2-d_3)L
$$ 
such that the double cover $\pi\colon V\to\mathbb{F}$ 
branched along $S$ is a canonical weak Fano $3$-fold, and let $X$ be the anti-canonical 
model of $V$. Observe that $-\left(K_{\mathbb{F}}+\frac{1}{2}S\right)\sim_\mathbb{Q}M$. 
Let $\mu\colon(\mathbb{F},\frac{1}{2}S)\to (\mathbb{F}',\frac{1}{2}S')$ be the 
anti-log-canonical model of $(\mathbb{F},\frac{1}{2}S)$ (i.e., $\mu$ is the ample 
model of $M$ and $S':=\mu_*S$). The pair $(\mathbb{F}',\frac{1}{2}S')$ is a (klt) 
$3$-dimensional log Fano pair. Recall that 
\begin{eqnarray*}
\delta\left(\mathbb{F},\frac{1}{2}S; M\right)
:=\inf_{{\substack{\mathbf{E}\text{: prime divisor}\\ \text{over } 
\mathbb{F}}}}\frac{A_{\mathbb{F},\frac{1}{2}S}(\mathbf{E})}{S(M; \mathbf{E})}, \\
\delta_p\left(\mathbb{F},\frac{1}{2}S; M\right)
:=\inf_{{\substack{\mathbf{E}\text{: prime divisor}\\ \text{over } \mathbb{F}\text{ with} 
\\ p\in c_X(\mathbf{E})}}}\frac{A_{\mathbb{F},\frac{1}{2}S}(\mathbf{E})}{S(M; \mathbf{E})}
\end{eqnarray*}
for any $p\in\mathbb{F}$, where 
\[
S\left(M;\mathbf{E}\right)=\frac{1}{M^3}\int_{0}^{\infty}\mathrm{vol}
\big(M-u\mathbf{E}\big)du.
\]

\begin{lemma}\label{lemma:diagram}
The following assertions hold:
\begin{enumerate}
\item[$(\mathrm{1})$]
The Fano 3-fold $X$ is K-semistable if and only if 
$\delta\left(\mathbb{F},\frac{1}{2}S; M\right)\geqslant 1$ holds.
\item[$(\mathrm{2})$]
Assume that there exists a reductive group $G\subset\mathrm{Aut}(X)$ such that 
\begin{itemize}
\item the projection  morphism $\mathbb{F}\to \mathbb{P}^1$ is $G$-equivariant, 
\item the surface  $S$ is $G$-invariant,
\item and the action $G\curvearrowright\mathbb{P}^1$ has no fixed point. 
\end{itemize}
If there exists a fiber $F_0\subset\mathbb{F}$ of $\mathbb{F}\to \mathbb{P}^1$ 
such that $\delta_p\left(\mathbb{F},\frac{1}{2}S; M\right)>1$ 
holds for any point $p\in F_0$, 
then $X$ is K-polystable. 
\end{enumerate}
\end{lemma}

\begin{proof}
Since $\mu^*(K_{\mathbb{F}'}+\frac{1}{2}S')=K_{\mathbb{F}}+\frac{1}{2}S$, we have 
$\delta\left(\mathbb{F}',\frac{1}{2}S'\right)
=\delta\left(\mathbb{F},\frac{1}{2}S; M\right)$, where the right hand side is the 
stability threshold of the log Fano pair $\left(\mathbb{F}',\frac{1}{2}S'\right)$ 
(see \cite[Definition 4.7]{Xu}). 
Moreover, from the construction, there is a finite double cover 
$\theta\colon X\to\mathbb{F}'$ with $K_X=\theta^*(K_{\mathbb{F}'}+\frac{1}{2}S')$. 
By \cite[Theorem 1.2]{LiuZhu}, the Fano 3-fold $X$ is K-semistable (resp., K-polystable) 
if and only if 
$\left(\mathbb{F}',\frac{1}{2}S'\right)$ is so. Thus we get the assertion (1). 

Let us consider the assertion (2). The group $G$ naturally acts on the pair 
$\left(\mathbb{F}',\frac{1}{2}S'\right)$. By \cite[Corollary 4.14]{Zhuang}, 
it is enough to show the inequality 
\[
\left(\frac{A_{\mathbb{F},\frac{1}{2}S}(\mathbf{E})}{S(M; \mathbf{E})}=\right)
\frac{A_{\mathbb{F'},\frac{1}{2}S'}(\mathbf{E})}{S(-(K_{\mathbb{F}'}+\frac{1}{2}S'); 
\mathbf{E})}>1
\]
for all $G$-invariant prime divisors $\mathbf{E}$ over $\mathbb{F}$. From the assumption, 
the center $c_{\mathbb{F}}(\mathbf{E})$ of $\mathbf{E}$ on $\mathbb{F}$ is horizontal 
with respect to $\mathbb{F}\to\mathbb{P}^1$. 
Thus, we obtain the desired inequality from the assumption. 
\end{proof}

\subsection{K-unstable cases}\label{subsection-unstable} This section concentrates on the proof of part $(\mathrm{3})$ of the Main Theorem.

For all possible $i\in\{1,2,3\}$ and $j\in\{1,2\}$, we set $D_i=\{x_i=0\}$ and $F_j=\{t_j=0\}$. 
Then 
$$
M\sim D_i +d_i F_j
$$
and $L\sim F_1\sim F_2$.

\begin{lemma}
\label{lemma:toric}
One has
\begin{eqnarray*}
S\left(M; D_i\right)&=&\frac{1}{4}\left(1+\frac{d_i}{d_1+d_2+d_3}\right), \\
S\left(M; F_j\right)&=&
\frac{d_1^2+d_2^2+d_3^2+d_1d_2+d_2d_3+d_3d_1}{4\left(d_1+d_2+d_3\right)}.
\end{eqnarray*}
\end{lemma}

\begin{proof}
Let $N^0=\mathbb{Z}^{3}$, and let $\Sigma$ be the~complete fan in $N^0_\mathbb{R}$ such 
that the~set of maximal dimensional cones of $\Sigma$ is equal to
\[
\Big\{\!\operatorname{Conv}(u_1,e_1,e_2), 
\operatorname{Conv}(u_1,e_1,e_3), \operatorname{Conv}(u_1,e_2,e_3), 
\operatorname{Conv}(u_2,e_1,e_2), \operatorname{Conv}(u_2,e_1,e_3), 
\operatorname{Conv}(u_2,e_2,e_3)\Big\}
\]
where
$e_1=(1,0,0)$, $e_2=(0,1,0)$, $e_3=(-1,-1,0)$, $u_1=(d_1-d_3,d_2-d_3,1)$, $u_2=(0,0,-1)$.
Then $\mathbb{F}$ is the~toric variety associated with the fan $\Sigma$.
Moreover, we have $D_i=V(e_i)$ and $F_j=V(u_j)$ for all possible $i\in\{1,2,3\}$ 
and $j\in\{1,2\}$.

The moment polytope in $M^0_\mathbb{R}=\mathbb{R}^{3}_{y_1y_2y_3}$ of 
the divisor $M$ is expressed as
\[
y_1\geqslant 0,\,\, y_2\geqslant 0,\,\, y_3\geqslant 0,\,\, y_1+y_2\leqslant 1,\,\, 
y_3\leqslant d_3+(d_1-d_3)y_1+(d_2-d_3)y_2,
\]
where $M^0$ is the~dual lattice of $N^0$.
By \cite[Corollary 7.7]{BlumJonsson} (see also \cite[Corollary 5.10]{Fujita2024}), we have
\begin{eqnarray*}
S\left(M;D_1\right)&=&\frac{3!}{d_1+d_2+d_3}
\int_{y_1=0}^1\int_{y_2=0}^{1-y_1}\int_{y_3=0}^{d_3+(d_1-d_3)y_1+(d_2-d_3)y_2}
y_1dy_3dy_2dy_1, \\
S\left(M; F_2\right)&=&\frac{3!}{d_1+d_2+d_3}
\int_{y_1=0}^1\int_{y_2=0}^{1-y_1}\int_{y_3=0}^{d_3+(d_1-d_3)y_1+(d_2-d_3)y_2}
y_3dy_3dy_2dy_1,
\end{eqnarray*}
which gives the required formulas for $S\left(M;D_1\right)$ and $S\left(M; F_2\right)$.
The remaining formulas can be obtained similarly.
\end{proof}

Using Lemma~\ref{lemma:toric}, we immediately prove part $(\mathrm{3})$ of Main Theorem:

\begin{corollary}
\label{corollary:kill-many}
If $X$ is contained in one of Families $H_{14}, H_{15}, H_{16}, H_{18},\ldots, H_{47}$, 
it is K-unstable.
\end{corollary}

\begin{proof}
By Lemma \ref{lemma:toric}, if
\[
4(d_1+d_2+d_3)<d_1^2+d_2^2+d_3^2+d_1d_2+d_2d_3+d_3d_1,
\]
then we have $\delta\left(\mathbb{F},\frac{1}{2}S; M\right)< 1$. Thus the~assertion 
follows from Lemma~\ref{lemma:diagram}.
\end{proof}

\subsubsection{Alternative Approach for K-unstable cases}\label{alternative-unstable}
We now illustrate some alternative ways to prove part $(\mathrm{3})$ of the Main Theorem. 
In some cases, this can also be done using $\alpha$-invariants.
Indeed, recall from \cite{CheltsovShramovUMN} that 
$$
\alpha(X):=\mathrm{sup}\left\{\lambda\in\mathbb{R}_{>0}\ \left|\ \aligned
&\text{the log pair}\ \left(X, \lambda D\right)\ \text{is log canonical for every }\\
&\text{effective $\mathbb{Q}$-divisor $D$ on $X$ such that}\ D\sim_{\mathbb{Q}}-K_{X}\\
\endaligned\right.\right\}.
$$
This definition immediately implies that
$$
\alpha(X)\leqslant\frac{1}{d_1}.
$$
On the~other hand, if $\alpha(X)<\frac{1}{4}$, then $X$ is K-unstable~\cite[Theorem 3.5]{FO}.
Thus, if $d_1\geqslant 5$, then $X$ is K-unstable
In particular, if $X$ is contained in one of Families $H_{24},H_{25},H_{26},\ldots,H_{46},H_{47}$, then $X$ is K-unstable.
A refinement of this approach gives the following result.

\begin{lemma}
\label{lemma:hyperelliptic-18-ldots-23}
Suppose that $X$ is contained in one of Families $H_{18},\ldots,H_{23}$. Then $X$ is K-unstable.
\end{lemma}

\begin{proof}
We have $d_1=4$, so that $\alpha(X)\leqslant\frac{1}{4}$. 
Suppose that $X$ is K-semistable. Then $\alpha(X)=\frac{1}{4}$,
and it follows from \cite[Proposition~3.1]{Ji17} that there exists a prime divisor $D\subset X$ such that
$$
-K_X\sim_{\mathbb{Q}} 4D.
$$
Let $\widetilde{D}$ be the strict transform of the divisor $D$ on the 3-fold $V$.
Since $d_3\ne 0$, the birational morphism $\phi$ is small, so $-K_V\sim_{\mathbb{Q}} 4\widetilde{D}$. Then it follows from the adjunction formula that
$$
-K_F\sim-K_V\big\vert_{F}\sim_{\mathbb{Q}} 4\widetilde{D}\big\vert_{F},
$$
where $F$ is a general fiber of the morphism $\eta$. Recall that $F$ is a del Pezzo surface of degree $2$ with at most Du Val singularities.
Thus, we have
$$
2=\big(-K_F\big)^2=4\big(-K_F\big)\cdot\widetilde{D}\big\vert_{F},
$$
which is impossible since $-K_F\cdot\widetilde{D}\vert_{F}$ is an integer.
\end{proof}

Similarly, one can occasionally show that a Fano 3-fold $X$ is K-unstable using the estimates for the normalized volume of singularities of K-semistable Fano varieties  established in \cite{Liu18,LiXu19}. That is, if $X$ is K-semistable, then
$$
(-K_X)^3\leqslant\frac{64}{27}\widehat{\mathrm{vol}}(P,V)
$$
for each point $P\in V$. Hence, using the description of the singularities of  $V$ given in \cite[Table~1]{ChPrSh}, we obtain the following result.

\begin{proposition}
\label{pro:volume}
If $X$ is contained in one of families $H_{15},H_{16},H_{20},\ldots,H_{47}$, then it is K-unstable.
\end{proposition}

\begin{proof}
It follows from \cite[Table~1]{ChPrSh} that in each case considered in this proposition the 3-fold $V$ is singular along a curve.
Let $P$ be a general point of this curve. Then $(P\in V)$ is locally isomorphic to
$$
(O\in Y)\times\mathbb{C}^1,
$$
where $(O\in Y)$ is a (two-dimensional) Du Val singularity.
Recall that $(O\in Y)$ is locally isomorphic to a quotient singularity $\mathbb{C}^2/\mathfrak{G}$ for a finite subgroup $\mathfrak{G}\subset\mathrm{SL}_2(\mathbb{C})$, and one of the following possibilities holds:
\begin{itemize}
\item if $(O\in Y)$ is of type $\mathbb{A}_n$, then $\mathfrak{G}$ is a cyclic group of order $n+1$,
\item if $(O\in Y)$ is of type $\mathbb{D}_n$, then $\mathfrak{G}$ is a binary dihedral group of order $4(n-2)$,
\item if $(O\in Y)$ is of type $\mathbb{E}_6$, then $\mathfrak{G}$ is a binary tetrahedral group of order $24$,
\item if $(O\in Y)$ is of type $\mathbb{E}_7$, then $\mathfrak{G}$ is a binary octahedral group of order $48$,
\item if $(O\in Y)$ is of type $\mathbb{E}_8$, then $\mathfrak{G}$ is a binary icosahedral group of order $120$.
\end{itemize}
On the other hand, using \cite[Theorem 1.3.2]{LiXu19}, we see that
$$
\widehat{\mathrm{vol}}(P,V)=\frac{27}{|\mathfrak{G}|}.
$$
Hence, if $X$ is K-semistable, then
$$
\frac{64}{|\mathfrak{G}|}\geqslant (-K_X)^3.
$$
Now, comparing the type of Du Val singularity $(O\in Y)$ given in~\cite[Table 1]{ChPrSh} with the~anticanonical degree $(-K_X)^3$ listed in \cite[Theorem 1.6]{ChPrSh},
we see that $X$ is not K-semistable.
\end{proof}

Family $H_{14}$ can be dealt with differently. 
Let $X$ be a general member of Family $H_{14}$. Since K-semistability is an open property, it is enough to show that $X$ is K-unstable, for example by showing that $\beta(F)<0$, which is equivalent to $S_X(F)>1$. 
By \cite[Lemma~4.1]{Fujita2019}, we have
$$
S_X\big(F\big)=\frac{1}{12}\int\limits_0^{3}\mathrm{vol}\big(-K_X-uF\big)du\geqslant \frac{1}{6}\int\limits_0^{3}\mathrm{vol}\big(M-uL\big)du,
$$
where the volume in the last integral is computed on $\mathbb{F}(a_1,a_2,a_3)$. Using \cite[Corollary 7.7]{BlumJonsson}, we get
$$
\frac{1}{6}\int\limits_0^{3}\mathrm{vol}\big(M-uL\big)du=\frac{25}{24},
$$
so $S_X(F)>1$. Then $\beta(F)<0$ and $X$ is K-unstable.

\subsection{K-stable cases}\label{subsection-stable}

To prove parts $(\mathrm{1})$ and $(\mathrm{2})$ of the Main Theorem in a uniform way, 
we use a new technique developed in \cite{Fujita2024}.
To do this, we need the following result, which uses notations introduced 
in \cite{Fujita2024}.

\begin{proposition}
\label{proposition:toric-lsc}
Let $Y$ be an $n$-dimensional projective
$\mathbb{Q}$-factorial toric variety, and let
\[
Y_\bullet\colon Y=Y_0\triangleright\cdots\triangleright Y_{n-1}\triangleright Y_n
\]
be a~torus invariant complete plt flag over $Y$ in the~sense of 
\cite[Definition 5.2]{Fujita2024}.
In particular, we have $Y_{n-1}\simeq \mathbb{P}^1$, and $Y_n$ is a torus invariant point 
in $Y_{n-1}$.
Let $Y''_n\in Y_{n-1}\setminus\{Y_n\}$ be the other torus invariant point.
Then, for any point $Y'_n\in Y_{n-1}\setminus\{Y_n,Y''_n\}$,
\[
Y'_\bullet\colon Y=Y_0\triangleright\cdots\triangleright Y_{n-1}\triangleright Y'_n
\]
is also a~complete plt flag over $Y$. Moreover, we have
\[
S\left(H; Y_1\triangleright\cdots\triangleright Y_{n-1}\triangleright 
Y'_n\right)\leqslant S\left(H; Y_1\triangleright\cdots\triangleright 
Y_{n-1}\triangleright Y_n\right)
\]
for any big $\mathbb{Q}$-divisor $H$ on $Y$. (See \cite[Definition 4.7]{Fujita2024} for 
the definition of the above values.)
\end{proposition}

\begin{proof}
From the~definition of plt flags, $Y'_\bullet$ is a~complete plt flag over $Y$. Since 
$Y_\bullet$ is a~toric plt flag over $Y$, there is a~smooth adequate dominant
\[
\left\{\gamma_k\colon\bar{Y}_k\to \tilde{Y}_k\right\}_{0\leqslant k\leqslant n-1}
\]
of $Y_\bullet$ with respect to $H$ in the sense of \cite[Definition 8.5]{Fujita2024} 
such that all $\gamma_k$ are toric morphisms.
Note that, from the~definition of adequate dominant, $\{\gamma_k\}_{0\leqslant 
k\leqslant n-1}$ is also an adequate dominant of $Y'_\bullet$.
Moreover, the~values $\{d_{lk}\}_{1\leqslant l\leqslant k\leqslant n-1}$
(hence also $\{g_{lk}\}_{1\leqslant l\leqslant k\leqslant n-1}$) in 
\cite[Definition 7.2]{Fujita2024}
with respect to the~dominant $\{\gamma_k\}_k$ of $Y_\bullet$ and the~dominant 
$\{\gamma_k\}_k$ of $Y'_\bullet$ are the~same.
Let
\begin{eqnarray*}
&&\mathbb{D}_k\subset\mathbb{R}^k_{>0}, \quad v_k(y_1,\dots,y_{k-1}), \quad
N_{l-1,k-1}(y_1,\dots,y_l)\\
\Bigl(\text{respectively, }&&
\mathbb{D}'_k\subset\mathbb{R}^k_{>0}, \quad v'_k(y_1,\dots,y_{k-1}), \quad
N'_{l-1,k-1}(y_1,\dots,y_l)
\Bigr)
\end{eqnarray*}
for $1\leqslant l\leqslant k\leqslant n$ be as in \cite[Definition 8.1]{Fujita2024}
with respects to $Y_\bullet$ and $H$ (respectively, $Y'_\bullet$ and $H$).
From the~definition, we have
\[
\mathbb{D}'_k=\mathbb{D}_k,\quad N'_{l-1,k-1}(y_1,\dots,y_l)=N_{l-1,k-1}(y_1,\dots,y_l)
\]
for any $1\leqslant l\leqslant k\leqslant n$, and
\[
v'_k(y_1,\dots,y_{k-1})=v_k(y_1,\dots,y_{k-1})
\]
for any $1\leqslant k\leqslant n-1$.
Since the~support of $N_{l-1,n-1}(y_1,\dots,y_l)$ lies only on
$\{Y_n, Y''_n\}$, we have
\[
v'_n(y_1,\dots,y_{n-1})\leqslant v_n(y_1,\dots,y_{n-1}).
\]
By \cite[Theorem 8.8]{Fujita2024}, we have
\begin{eqnarray*}
&&S\left(H; Y_1\triangleright\cdots\triangleright Y_{n-1}\triangleright Y'_n\right)\\
&=&\frac{n!}{\operatorname{vol}(H)}
\int_{(y_1,\dots,y_n)\in\mathbb{D}_n}\left(y_n+v'_n(y_1,\dots,y_{n-1})
+\sum_{l=1}^{k-1}g_{l,k}\left(y_l+v_l(y_1,\dots,y_{l-1})\right)\right)d\vec{y}\\
&\leqslant&\frac{n!}{\operatorname{vol}(H)}
\int_{(y_1,\dots,y_n)\in\mathbb{D}_n}\left(y_n+v_n(y_1,\dots,y_{n-1})
+\sum_{l=1}^{k-1}g_{l,k}\left(y_l+v_l(y_1,\dots,y_{l-1})\right)\right)d\vec{y}\\
&=&S\left(H; Y_1\triangleright\cdots\triangleright Y_{n-1}\triangleright Y_n\right).
\end{eqnarray*}
Thus, we get the~assertion.
\end{proof}

Using Proposition \ref{proposition:toric-lsc} and the results in 
\cite[\S 5]{Fujita2024}, we obtain the next two lemmas. 

\begin{lemma}
\label{lemma:F1}
Take any $i_1,i_2,i_3$ with $\{i_1,i_2,i_3\}=\{1,2,3\}$, and let $l_{i_1}=\{t_2=x_{i_1}=0\}$.
Then
\[
S\left(M; F_2\triangleright l_{i_1}\right)=\frac{1}{4}
\left(1+\frac{d_{i_1}}{d_1+d_2+d_3}\right).
\]
Moreover, for any point $p\in l_{i_1}$, we have
\[
S\left(M; F_2\triangleright l_{i_1}\triangleright p\right)\leqslant
\frac{1}{4}\left(1+\frac{\max\{d_{i_2}, d_{i_3}\}}{d_1+d_2+d_3}\right).
\]
\end{lemma}

\begin{proof}
We only consider the~case $i_1=3$. Let $p_1=\{t_2=x_2=x_3=0\}$ and 
$p_2=\{t_2=x_1=x_3=0\}$.
Then, in the~terminology of \cite[Definition 5.2]{Fujita2024},
the complete flag
\[
X\triangleright F_2\triangleright l_3\triangleright p_2
\]
is associated with
\[
v_1=u_2,\quad v_2=\pi_1(e_3), \quad v_3=\pi_2\circ\pi_1(e_1),
\]
where $\pi_k\colon N^{k-1}\to N^k=N^{k-1}/\mathbb{Z} v_k$ are as in 
\cite[Definition 5.1]{Fujita2024}.
By \cite[Theorem 5.12]{Fujita2024}, we have
\begin{align*}
S\left(M; F_2\triangleright l_3\right)&=S\left(M; D_3\right),\\
S\left(M; F_2\triangleright l_3\triangleright p_2\right)&=S\left(M; D_1\right),\\
S\left(M; F_2\triangleright l_3\triangleright p_1\right)&=S\left(M; D_2\right).
\end{align*}
Thus, the~assertion follows from Lemma~\ref{lemma:toric} and 
Proposition~\ref{proposition:toric-lsc}.
\end{proof}

\begin{lemma}
\label{lemma:F2}
Fix coprime positive integers $a_1,a_2$,
let $\tilde{F}_2\to F_2$ be the toric weighted blowup of the point $\{x_1=x_2=0\}$ 
along local coordinates $(x_1,x_2)$ with ${\rm wt}(x_1,x_2)=(a_1,a_2)$,
and let $E$ be the~exceptional curve of this blowup.
Then
\[
S\left(M; F_2\triangleright E\right)=\frac{1}{4}\left(a_1+a_2+
\frac{a_1d_1+a_2d_2}{d_1+d_2+d_3}\right).
\]
Moreover, for any point $q\in E$, we have
\[
S\left(M; F_2\triangleright E\triangleright q\right)\leqslant
\max\left\{\frac{1}{4a_2}\left(1+\frac{d_1}{d_1+d_2+d_3}\right),\quad
\frac{1}{4a_1}\left(1+\frac{d_2}{d_1+d_2+d_3}\right)\right\}.
\]
\end{lemma}

\begin{proof}
For $i\in\{1,2\}$, let $q_i$ be the intersection point of the curve $E$ and the~strict 
transform of the~curve $\{t_2=x_i=0\}$.
Then, in the~terminology of \cite[Definition 5.2]{Fujita2024},
the complete flag
\[
X\triangleright F_2\triangleright E\triangleright q_1
\]
is associated with
\[
v_1=u_2,\quad v_2=a_1\pi_1(e_1)+a_2\pi_1(e_2), \quad v_3
=\frac{1}{a_2}\pi_2\circ\pi_1(e_1).
\]
Again, it follows from \cite[Theorem 5.12]{Fujita2024} that
\begin{align*}
S\left(M; F_2\triangleright E\right)&=a_1S\left(M; D_1\right)+a_2S\left(M; D_2\right),\\
S\left(M; F_2\triangleright E\triangleright q_1\right)
&=\frac{S\left(M; D_1\right)}{a_2},\\
S\left(M; F_2\triangleright E\triangleright q_2\right)&=\frac{S\left(M; D_2\right)}{a_1}.
\end{align*}
Thus, the~assertion follows from Lemma~\ref{lemma:toric} and 
Proposition~\ref{proposition:toric-lsc}.
\end{proof}

Now, we are ready to prove the following proposition, which implies part $(\mathrm{1})$ of 
Main Theorem, since K-stability is an open property \cite[Theorem 1.1]{BlumLiuXu}.

\begin{proposition}
\label{proposition:marseille}
Suppose that one of the following six cases holds:
\begin{itemize}
\item $X$ is contained in Family $H_5$ and 
$S=\{(t_1^6+t_2^6)x_1^4+x_1x_3^3+t_1t_2x_2^4+x_2^2x_3^2=0\}$;
\item $X$ is contained in Family $H_7$ and 
$S=\{(t_1^4+t_2^4)x_1^4+t_1^2t_2^2x_2^4+x_1^2x_3^2+x_2^2x_3^2=0\}$;
\item $X$ is contained in Family $H_8$ and 
$S=\{(t_1^2+t_2^2)x_1^4+t_1t_2x_2^4+x_1^2x_3^2+x_2^2x_3^2=0\}$;
\item $X$ is contained in Family $H_{11}$ and 
$S=\{(t_1^8+t_2^8)x_1^4+t_1t_2x_1x_3^2+x_1x_2x_3^2+x_2^4=0\}$;
\item $X$ is contained in Family $H_{12}$ and 
$S=\{x_1\left((t_1^6+t_2^6)x_1^3+x_2^3+x_3^3\right)=0\}$;
\item $X$ is contained in Family $H_{13}$ and 
$S=\{(t_1^6+t_2^6)x_1^4+x_1^2x_3^2+t_1t_2x_2^4+x_2^3x_3=0\}$.
\end{itemize}
Then $X$ is K-stable.
\end{proposition}

\begin{proof}
Arguing as in \cite{CheltsovShramovPrzyjalkowski} and using \cite{CheltsovProkhorov}, it is not difficult to see that the group $\mathrm{Aut}(X)$ is finite. Thus, in order to prove that $X$ is K-stable, it is enough to show that it is K-polystable.
Observe that, in each case, there is a dihedral group $G$ acting equivariantly  on $\left(\mathbb{F},\frac{1}{2}S\right)\to\mathbb{P}^1$ and the action of $G$ 
to $\mathbb{P}^1$ has no fixed points. 
For example, if $X$ is contained in Family $H_5$, then $G$ is taken to be the group 
generated by the following transformations:
\begin{align*}
(x_1:x_2:x_3; t_1:t_2)&\mapsto(x_1:x_2:x_3; t_2:t_1),\\
(x_1:x_2:x_3; t_1:t_2)&\mapsto\left(x_1:x_2:x_3; \omega_6t_1:\omega_6^{-1}t_2\right), 
\end{align*}
where $\omega_6$ is a primitive sixth root of unity. 
Take a general fiber $F_0\subset\mathbb{F}$ of $\mathbb{F}\to\mathbb{P}^1$. 
By Lemma \ref{lemma:diagram} (2), 
it is enough to show the~inequality
$$
\delta_p\left(\mathbb{F},\frac{1}{2}S; M\right)>1
$$
for every point $p\in F_0$.

First, we consider the cases when $X$ is contained in Families 
$H_5$, $H_7$, $H_8$, $H_{11}$, $H_{13}$.
In these cases, the~restriction $S|_{F_0}$ is an irreducible plane quartic curve in 
$F_0\simeq\mathbb{P}^2$, which is smooth away from the point $p_3=\{x_1=x_2=0\}\cap F_0$.
The~singularity of the curve $S|_{F_0}$ at the point $p_3$ is:
\begin{itemize}
\item a smooth point in the case when $X$ is contained in Family $H_5$,
\item a nodal singularity in the case when $X$ is contained in one of Families 
$H_7$, $H_8$, $H_{11}$,
\item a simple cusp in the case when $X$ is contained in Family $H_{13}$.
\end{itemize}

Suppose that $p\neq p_3$. Take a~general line $l\subset F_0\simeq\mathbb{P}^2$ 
that contains $p$.
We can choose $l$ such that it intersects the curve $S|_{F_0}$ transversely.
Note that the values 
\[
S\left(M; F_0\right), \quad S\left(M; F_0\triangleright l\right),\quad 
S\left(M; F_0\triangleright l\triangleright p\right)
\]
do not depend on the~choice of the surface $S\subset\mathbb{F}$. Moreover, there is 
an automorphism of the scroll $\mathbb{F}$ such that
the fiber $F_0$ is mapped to $F_2$, and $l$ is mapped to the line $l_3$ in the notation 
of Lemma~\ref{lemma:F1}.
Thus, by Lemma \ref{lemma:F1}, we have
\begin{eqnarray*}
S\left(M; F_0\right)&=&\frac{\sum\limits_{1\leqslant i_1\leqslant i_2\leqslant 3}
d_{i_1}d_{i_2}}{d_1+d_2+d_3},\\
S\left(M; F_0\triangleright l\right)&=&
\frac{1}{4}\left(1+\frac{d_3}{d_1+d_2+d_3}\right),\\
S\left(M; F_0\triangleright l\triangleright p\right)&\leqslant&
\frac{1}{4}\left(1+\frac{d_1}{d_1+d_2+d_3}\right).
\end{eqnarray*}
On the other hand, in the~notation of \cite[Definition 2.10]{Fujita2024}, we have
\[
A_{\mathbb{F},\frac{1}{2}S}(F_0)=1, \quad A_{\mathbb{F},
\frac{1}{2}S}(F_0\triangleright l)=1,\quad
A_{\mathbb{F},\frac{1}{2}S}(F_0\triangleright l\triangleright p)\geqslant\frac{1}{2}. 
\]
Thus, it follows from \cite[Theorem 3.2]{AbbanZhuang} that
\[\delta_p\left(\mathbb{F},\frac{1}{2}S; M\right)\geqslant\min\left\{
\frac{A_{\mathbb{F},\frac{1}{2}S}(F_0)}{S\left(M; F_0\right)}, \quad
\frac{A_{\mathbb{F},\frac{1}{2}S}(F_0\triangleright l)}{S\left(M;
F_0\triangleright l\right)}, \quad
\frac{A_{\mathbb{F},\frac{1}{2}S}(F_0\triangleright l\triangleright 
p)}{S\left(M;F_0\triangleright l\triangleright p\right)}\right\}>1
\]
in each case. Hence we may assume that $p=p_3$ from now on. Take the weighted blowup 
$\tilde{F}_0\to F_0$ at the point $p_3$ with weights $(a_1,a_2)$ with local
coordinates $(x_1,x_2)$ such that
\[
(a_1,a_2)=\begin{cases}
(1,1) & \text{if $X$ is contained in one of Families $H_5$, $H_7$, $H_8$, $H_{11}$}, \\
(3,2) & \text{if $X$ is contained in Family $H_{13}$}.
\end{cases}
\]
Let $E\subset\tilde{F}_0$ be the exceptional divisor of the blowup, and let
$q\in E$ be any closed point. We have
\begin{eqnarray*}
A_{\mathbb{F},\frac{1}{2}S}(F_0)&=&1, \\
A_{\mathbb{F},\frac{1}{2}S}(F_0\triangleright E)&=&\begin{cases}
\frac{3}{2} & \text{if $X$ is contained in Family $H_5$}, \\
1 & \text{if $X$ is contained in one of Families $H_7$, $H_8$, $H_{11}$}, \\
2 & \text{if $X$ is contained in Family $H_{13}$},
\end{cases}\\
A_{\mathbb{F},\frac{1}{2}S}(F_0\triangleright E\triangleright q)&\geqslant&\begin{cases}
\frac{1}{2} & \text{if $X$ is contained in one of Families 
$H_5$, $H_7$, $H_8$, $H_{11}$}, \\
\frac{1}{3} & \text{if $X$ is contained in Family $H_{13}$}.
\end{cases}
\end{eqnarray*}
On the other hand, after twisting by 
an automorphism of $\mathbb{F}$, similar to the~above arguments, we get
\begin{eqnarray*}
S\left(M; F_0\triangleright E\right)
&=&\frac{1}{4}\left(a_1+a_2+\frac{a_1d_1+a_2d_2}{d_1+d_2+d_3}\right),\\
S\left(M; F_0\triangleright E\triangleright q\right)
&\leqslant&
\max\left\{\frac{1}{4a_2}\left(1+\frac{d_1}{d_1+d_2+d_3}\right),\quad
\frac{1}{4a_1}\left(1+\frac{d_2}{d_1+d_2+d_3}\right)\right\}
\end{eqnarray*}
by Lemma \ref{lemma:F2}.
In each case, again by \cite[Theorem 3.2]{AbbanZhuang}, we have
$\delta_{p_3}\left(\mathbb{F},\frac{1}{2}S;M\right)>1$.

Now, we consider the~remaining case when $X$ is contained in Family $H_{12}$.
Then
$S|_{F_0}=l_1+C'$,
where $l_1$ is the line in $F_0\simeq\mathbb{P}^2$ that is cut out by $x_1=0$, and 
$C'$ is a~smooth plane cubic curve that transversely intersects $l_1$.
We consider the case $p\not\in l_1$. 
Let $l$ be a general line in $F_0\simeq\mathbb{P}^2$ passing through $p$.
Then, as before, we have 
\[
S\left(M; F_0\right)=\frac{9}{10}, \quad 
S\left(M; F_0\triangleright l\right)=\frac{3}{10},
\quad S\left(M; F_0\triangleright l\triangleright p\right)\leqslant\frac{2}{5}
\]
by Lemma \ref{lemma:F1}.
Moreover, we have 
\[
A_{\mathbb{F},\frac{1}{2}S}(F_0)=1, \quad 
A_{\mathbb{F},\frac{1}{2}S}(F_0\triangleright l)=1, \quad 
A_{\mathbb{F},\frac{1}{2}S}(F_0\triangleright l\triangleright p)\geqslant\frac{1}{2},
\]
so that $\delta_p\left(\mathbb{F},\frac{1}{2}S;M\right)>1$ by 
\cite[Theorem 3.2]{AbbanZhuang}.
Thus we may assume that $p\in l_1$.
Then we have 
\[
S\left(M; F_0\triangleright l_1\right)=\frac{2}{5},\quad 
S\left(M; F_0\triangleright l_1\triangleright p\right)\leqslant\frac{3}{10}
\]
by Lemma \ref{lemma:F1},
and
\[
A_{\mathbb{F},\frac{1}{2}S}(F_0\triangleright l_1)=\frac{1}{2}, \quad  
A_{\mathbb{F},\frac{1}{2}S}(F_0\triangleright l_1\triangleright p)\geqslant\frac{1}{2},
\]
so that $\delta_p\left(\mathbb{F},\frac{1}{2}S;M\right)>1$ by 
\cite[Theorem 3.2]{AbbanZhuang}. This completes the~proof of the proposition.
\end{proof}

\subsubsection{Alternative Approach for K-stable cases}\label{alternative-stable}
Certain cases above can be proven K-stable with more traditional methods. For example, Tian's criterion can be deployed provided that
$$
\alpha_G(X)>\frac{3}{4},
$$
where $G$ is a reductive subgroup in $\mathrm{Aut}(X)$ and
$$
\alpha_G(X):=\mathrm{sup}\left\{\lambda\in\mathbb{R}_{>0}\ \left|\ \aligned
&\text{the pair}\ \left(X, \frac{\lambda}{m} \mathcal{M}\right)\ \text{is log canonical for every }m\in\mathbb{Z}_{>0} \\
&\text{and for every $G$-invariant subsystem }\ \mathcal{M}\subset |-mK_{X}|\\
\endaligned\right.\right\}.
$$
Moreover, if the group $\mathrm{Aut}(X)$ is finite, the K-stability and K-polystability coincide as mentioned around the beginning of the proof of Proposition \ref{proposition:marseille}. The use of Tian's criterion is illustrated in the following example.

\begin{example}[Family $H_5$]
\label{example:H5}
Suppose that $d_1=2$, $d_2=1$, $d_3=0$, and
$$
S=\big\{(t_1^6+t_2^6)x_1^4+x_1x_3^3+t_1t_2x_2^4+x_2^2x_3^2=0\big\}\subset\mathbb{F}(2,1,0).
$$
Then $X$ is contained in Family $H_5$, 
\mbox{$(-K_X)^3=6$}, both $V$ and $S$ are smooth, and $\phi$ is a~small morphism that contracts a~smooth rational curve
to an~ordinary double point of the Fano 3-fold~$X$. 
Moreover, arguing as in \cite{CheltsovShramovPrzyjalkowski}, we see that  $\mathrm{Aut}(X)$ is finite.
Let $G\subset\mathrm{Aut}(V)$ be the~subgroup generated by the~Galois involution of $\pi$,
and the~lifts of the~automorphisms of $\mathbb{F}(2,1,0)$ given by
\begin{align*}
(x_1:x_2:x_3;t_1:t_2)&\mapsto (x_1:x_2:x_3;\omega_6t_1:\omega_6^5t_2),\\
(x_1:x_2:x_3;t_1:t_2)&\mapsto (x_1:x_2:x_3;t_2:t_1),\\
(x_1:x_2:x_3;t_1:t_2)&\mapsto (x_1:-x_2:x_3;t_1:t_2).
\end{align*}
where $\omega_6$ is a~primitive sixth root of unity.
Then we have the following  $G$-equivariant  diagram:
$$
\xymatrix{
&V\ar[ld]_{\eta}\ar[rd]^{\phi}&&U\ar[ld]_{\psi}\ar[rd]^{\nu}&\\
\mathbb{P}^1&&X&&Y}
$$
where $\psi$ is a~small resolution of $X$,
$Y$ is a~sextic hypersurface in $\mathbb{P}(1^3,2,3)$ that has one ordinary double point, and $\nu$ is the blowup of this point.
Hence, we can consider $G$ as a subgroup in $\mathrm{Aut}(X)$. Then $\alpha_G(X)=\alpha_G(V)$ by \cite[Lemma~1.47]{Book},
and $\alpha_G(V)\geqslant\alpha(F)$ by \cite[Corollary~1.56]{Book}, where $F$ is the generic fiber of $\eta$. 
On the other hand, it follows from \cite{CheltsovGAFA,CheltsovParkWonJEMS} that $\alpha(F)\geqslant\frac{5}{6}$,
because $|-K_{F}|$ does not contain curves with tachnodal singularities.
Hence, we conclude that $\alpha_G(X)\geqslant\frac{5}{6}$, and $X$ is K-stable.
\end{example}

Quite often we are unable to show that $\alpha_G(X)>\frac{3}{4}$, but we can show that
$\alpha_{G,Z}(X)>\frac{3}{4}$ for every irreducible $G$-invariant curve $Z\subset X$ that satisfies certain geometric conditions,
where $\alpha_{G,Z}(X)$ is the local analogue of the $\alpha$-invariant defined in \cite[\S~1.4]{Book}.
Together with equivariant valuative criterion and Abban--Zhuang theory \cite{AbbanZhuang,Book},
this enables us to show that $X$ is K-stable in many cases. Let us consider one example.

\begin{example}[Family $H_7$]
\label{example:H7}
Suppose that $d_1=2$, $d_2=2$, $d_3=0$, and
$$
S=\big\{(t_1^4+t_2^4)x_1^4+t_1^2t_2^2x_2^4+x_1x_2x_3^2=0\big\}\subset\mathbb{F}(2,2,0).
$$
Then $X$ is a hyperelliptic Fano 3-fold in Family $H_7$ with $(-K_X)^3=8$, the group $\mathrm{Aut}(X)$ is finite,
the surface $S$ is singular along the~curve $\{x_1=x_2=0\}\subset\mathbb{F}(2,2,0)$,
it has an ordinary double singularity at general point of this curve,
and $\phi$ contracts its preimage on $V$, which is $\mathrm{Sing}(V)$.
Let $G$ be the~subgroup in $\mathrm{Aut}(V)$ generated by the~Galois involution of the~double cover~$\pi$,
and the~lifts of the~automorphisms of $\mathbb{F}(2,2,0)$ given by
\begin{align*}
(x_1:x_2:x_3;t_1:t_2)&\mapsto (x_1:x_2:x_3;it_1:-it_2),\\
(x_1:x_2:x_3;t_1:t_2)&\mapsto (x_1:x_2:x_3;t_2:t_1),\\
(x_1:x_2:x_3;t_1:t_2)&\mapsto (ix_1:-ix_2:x_3;t_1:t_2).
\end{align*}
Then $\phi$ and $\eta$ are $G$-equivariant,
and we can consider $G$ as a~subgroup in $\mathrm{Aut}(X)$.
Let $C=\mathrm{Sing}(V)$. Then $C$ is a section of the~fibration~$\eta$,
and $F\cap C$ is an ordinary double singularity of the surface $F$.
Let~$\mathbf{E}$ be a~$G$-invariant prime divisor over $X$,
and let $Z$ be its center on $V$.
Then $\eta(Z)=\mathbb{P}_{t_1,t_2}^1$, since $G$ does not fix points in $\mathbb{P}^1_{t_1,t_2}$.
In particular, we have $Z\cap F\ne\varnothing$.
Let $P$ be a~point in $Z\cap F$.
Then \cite[Corollary 1.44]{Book} gives
$$
\frac{3}{4}\frac{A_X(\mathbf{E})}{S_X(\mathbf{E})}\geqslant\alpha_{G,Z}(X)\geqslant\alpha_P(F),
$$
where $\alpha_{P}(F)$ is the $\alpha$-function of the del Pezzo surface $F$ \cite{CheltsovParkWonJEMS, ParkWon,ParkWon2011}.
So, if $\beta(\mathbf{E})\leqslant 0$, then $\alpha_P(F)\leqslant \frac{3}{4}$,
so it follows from \cite{CheltsovParkWonJEMS, ParkWon,ParkWon2011} that one of the following cases holds:
\begin{enumerate}
\item either $P$ is the~singular point of the~surface $F$,
\item or $P$ is a~smooth point of $F$ and $|-K_F|$ contains a~tacknodal curve singular at $P$.
\end{enumerate}
In both cases, we see that $Z$ is a~curve.
Moreover, arguing as in the~proof of  \cite[Theorem 1.52]{Book}, we see that $Z$ must be  a~section of the~del Pezzo fibration $\eta$.
Therefore, if $\beta(\mathbf{E})\leqslant 0$, then $Z=C$, because $C$ is the unique $G$-invariant section of $\eta$.
We also have a commutative diagram
$$
\xymatrix{
Y\ar[d]_{f}\ar[rr]^{g}&&U\ar[d]^{\nu}\\
V\ar@{-->}[rr]&&\mathbb{P}^1_{x_1,x_2}}
$$
where $f$ is the blow of the curve $C$,
$g$ is a~contraction of the~$f$-exceptional surface to a~singular curve of the~3-fold $U$,
$\nu$ is a~fibration into del Pezzo surfaces of degree $4$,
and the dashed arrow is the composition of $\pi$ with the~map $\mathbb{F}(2,2,0)\dasharrow\mathbb{P}^1_{x_1,x_2}$ given by
\[
(x_1:x_2:x_3;t_1:t_2)\mapsto(x_1:x_2).
\]
Let $E$ be the~exceptional surface of the~morphism $f$,
and let $\widetilde{Z}$ be the center on $Y$ of the divisor~$\mathbf{E}$.
Then $Y$ is smooth, $E\cong\mathbb{P}^1\times\mathbb{P}^1$,
and $\widetilde{Z}$ is a~curve in $E$ such that $f(\widetilde{Z})=C$, since $\beta(E)>0$, because
$$
S_{X}(E)=\frac{1}{8}\int\limits_0^{1}\big(f^*(-K_V)-uE\big)^3du=\frac{1}{8}\int\limits_{0}^{1}8-8u^3du=\frac{3}{4}<1=A_X(E).
$$
Now, we apply Abban--Zhuang theory \cite{AbbanZhuang,Book} to show that $\beta(\mathbf{E})>0$.
Namely, following \cite{AbbanZhuang,Book}, let
$$
S\big(W^E_{\bullet,\bullet};\widetilde{Z}\big)=\frac{3}{(-K_X)^3}\int\limits_0^{1}\int\limits_0^{\infty}\mathrm{vol}\Big(\big(f^*(-K_V)-uE\big)\big\vert_{E}-v\widetilde{Z}\Big)dvdu.
$$
Note that value $S(W^E_{\bullet,\bullet};\widetilde{Z})$ in \cite{Book} is nothing but the value $S(-K_X; E\triangleright \widetilde{Z})$ in \cite{Fujita2024}. It follows from \cite{AbbanZhuang,Book} that
$$
\frac{A_X(\mathbf{E})}{S_X(\mathbf{E})}\geqslant\min\Bigg\{\frac{1}{S_X(E)},\frac{1}{S\big(W^E_{\bullet,\bullet};\widetilde{Z}\big)}\Bigg\}.
$$
But we already know that $S_X(E)=\frac{3}{4}$. Moreover, we compute
$$
S\big(W^E_{\bullet,\bullet};\widetilde{Z}\big)\leqslant
\frac{3}{(-K_X)^3}\int\limits_0^{1}\int\limits_0^{\infty}\mathrm{vol}\Big(\big(f^*(-K_V)-uE\big)\big\vert_{E}-v\mathbf{s}\Big)dvdu=\frac{3}{8}\int\limits_0^{1}\int\limits_0^{2u}4u(2u-v)dvdu=\frac{3}{4},
$$
where $\mathbf{s}$ is a curve in $E\cong\mathbb{P}^1\times\mathbb{P}^1$ such that $\mathbf{s}^2=0$ and $f(\mathbf{s})=C$.
So, $\beta(\mathbf{E})>0$, and $X$ is K-stable.
\end{example}

Similarly, one can produce symmetric K-stable singular members of Families $H_8,H_{11},H_{12},H_{13}$ and argue as above to prove their K-stability. We leave this for further investigation by future generations.

\subsection{K-polystable cases}\label{subsection-polystable}

Finally, we prove $(\mathrm{2})$ of the Main Theorem, which follows from the next proposition.

\begin{proposition}
\label{proposition:10-17}
Suppose that one of the following two cases holds:
\begin{itemize}
\item $X$ is contained in Family $H_{10}$ and $S=\{x_1(t_1x_2^3+t_2x_3^3)=0\}$;
\item $X$ is contained in Family $H_{17}$ and $S=\{x_1(x_2^3+x_3^3)=0\}$.
\end{itemize}
Then $X$ is K-polystable.
\end{proposition}

\begin{proof}
As in the proof of Proposition~\ref{proposition:marseille},
we see that the Fano 3-fold $X$ is K-polystable if and only if the~anti-canonical model 
$(\mathbb{F}',\frac{1}{2}S')$ of the~pair $(\mathbb{F},\frac{1}{2}S)$ is K-polystable.
On the other hand, the surface $S$ is reducible and can be decomposed into 
$S=S''+D_1$, and Lemma~\ref{lemma:F1} gives
\[
A_{\mathbb{F},\frac{1}{2}S}(D_1)\leqslant\frac{1}{2}=S\left(M;D_1\right).
\]
This implies that the log Fano pair $(\mathbb{F}',\frac{1}{2}S')$ is not~\mbox{K-stable}.

Suppose that $X$ is contained in Family $H_{10}$ and $S=\{x_1(t_0x_2^3+t_1x_3^3)=0\}$.
Then $\operatorname{Aut}\left(\mathbb{F},\frac{1}{2}S\right)$ contains a two-dimensional 
torus $\mathbb{T}\simeq\mathbb{G}_m^2$ that acts on $\mathbb{F}$ as follows:
\begin{eqnarray*}
\left(x_1:x_2:x_3; t_1:t_2\right) &\mapsto&
\left(\lambda_1x_1:x_2:x_3; t_1:t_2\right),\\
\left(x_1:x_2:x_3; t_1:t_2\right) &\mapsto&
\left(x_1:\lambda_2x_2:x_3; t_1:\lambda_2^3t_2\right).
\end{eqnarray*}
Moreover, the log Fano pair $(\mathbb{F}',\frac{1}{2}S')$ 
has a~vanishing $\mathbb{T}$-Futaki character (see \cite[\S 2.2]{Xu}).
Indeed, let $\mathbb{E}$ be the~exceptional divisor over $\mathbb{F}$ of the~weighted blowup with weights $(1,3)$ along local coordinates $(x_2,t_2)$.
Then 
\[
A_{\mathbb{F},\frac{1}{2}S}(\mathbb{E})=\frac{5}{2},\quad
S\left(M;\mathbb{E}\right)=S\left(M;D_2\right)+3S\left(M;F_1\right)=\frac{5}{2}
\]
by Lemma~\ref{lemma:toric}. Moreover, the~quotient map 
$\mathbb{F}\dasharrow\mathbb{P}^1$ by the torus $\mathbb{T}$ is given by
\[
\left(x_1:x_2:x_3; t_1:t_2\right)\mapsto\left(t_1x_2^3:t_2x_3^3\right).
\]
We claim that any vertical prime divisor $E$ on $\mathbb{F}$ satisfies  
$A_{\mathbb{F},\frac{1}{2}S}(E)>S\left(M; E\right)$.
Indeed, we have
\begin{eqnarray*}
A_{\mathbb{F},\frac{1}{2}S}(E)&=&\begin{cases}
\frac{1}{2} & \text{if }E=S'', \\
1 & \text{otherwise},
\end{cases}\\
S\left(M; E\right)&\leqslant& S\left(M; F_1\right)=\frac{3}{4}, \\
S\left(M; S''\right)&\leqslant& S\left(M; D_3\right)=\frac{1}{4}.
\end{eqnarray*}
Thus, the log Fano pair $(\mathbb{F}',\frac{1}{2}S')$ is K-polystable as in 
\cite[Theorem 1.31]{Book}.

Now, we suppose that $X$ is contained in Family $H_{17}$ and $S=\{x_1(x_2^3+x_3^3)=0\}$.
As above, we see that the group $\operatorname{Aut}\left(\mathbb{F},\frac{1}{2}S\right)$ 
contains a two-dimensional torus $\mathbb{T}\simeq\mathbb{G}_m^2$ that acts 
on $\mathbb{F}$ as follows:
\begin{eqnarray*}
\left(x_1:x_2:x_3; t_1:t_2\right) &\mapsto&\left(\lambda_1x_1:x_2:x_3; t_1:t_2\right),\\
\left(x_1:x_2:x_3; t_1:t_2\right) &\mapsto&\left(x_1:x_2:x_3; t_1:\lambda_2t_2\right).
\end{eqnarray*}
Then, since $S\left(M; F_1\right)=1$, the log Fano pair $(\mathbb{F}',\frac{1}{2}S')$ 
has vanishing $\mathbb{T}$-Futaki character.
Moreover, the~quotient map $\mathbb{F}\dasharrow\mathbb{P}^1$ by $\mathbb{T}$ is given by
\[
\left(x_1:x_2:x_3; t_1:t_2\right)\mapsto\left(x_2:x_3\right).
\]
As above, we see that $A_{\mathbb{F},\frac{1}{2}S}(E)>S\left(M; E\right)$ for any 
vertical prime divisor $E$ on $\mathbb{F}$,
which implies that the log Fano pair $(\mathbb{F}',\frac{1}{2}S')$ is K-polystable 
as in \cite[Theorem 1.31]{Book}.
\end{proof}

This completes the proof of the Main Theorem.

\subsubsection{Alternative Approach for K-polystable cases}\label{alternative-polystable}
Suppose that $X$ is the hypersurfaces \eqref{equation:H17}. We show that $X$ is K-polystable, using other techniques.  Similar argument can be used for \eqref{equation:H10}. We have $d_1=4$, $d_2=0$, $d_3=0$, and
$$
S=\big\{x_1(x_2^3+x_3^3)=0\big\}\subset\mathbb{F}(4,0,0).
$$
Here, we use notations introduced in Section~\ref{section:introduction}.
Let $G$ be the subgroup in $\mathrm{Aut}(X)$ generated by
\begin{align*}
(x:y:z:t:w)&\mapsto(x:y:t:z:w),\\
(x:y:z:t:w)&\mapsto(x:y:\omega_3t:-\omega_3(z+t):w),\\
(x:y:z:t:w)&\mapsto(ax+by:cx+dy:z:t:w),
\end{align*}
where $\omega_3$ is a~primitive cube root of unity,
and $a,b,c,d$ are complex numbers such that $ad-bc\ne 0$.
Then we have the~following $G$-equivariant commutative diagram:
$$
\xymatrix{
V\ar@{->}[d]_{\phi}\ar@{->}[drr]^{\eta}\ar@{->}[rr]^{\pi}&&\mathbb{F}(4,0,0)\ar@{->}[d]\\%
X\ar@{-->}[rr]^{\rho} && \mathbb{P}^1_{t_1,t_2}}
$$
where $\rho$ is the rational map given by $[x:y:z:t:w]\mapsto[x:y]$,
and $\mathbb{F}(4,0,0)\to\mathbb{P}^1_{t_1,t_2}$ is the~natural projection.
We also have the~following $G$-equivariant commutative diagram:
$$
\xymatrix{
&U\ar@{->}[dl]_{\varphi}\ar@{->}[dr]^{\vartheta}\\%
X\ar@{-->}[rr]^{\varrho} && \mathbb{P}^1}
$$
where $\varrho$ is the~map given by $[x:y:z:t:w]\mapsto[z:t]$,
$\vartheta$ is a~morphism, and $\varphi$ is a~weighted blow up of the~curve $\{z=t=w=0\}$.
Let $E_\phi$ and $E_{\varphi}$ be the exceptional surfaces of $\phi$ and $\varphi$, respectively.
Then $E_\phi$ and $E_\varphi$  are the~only $G$-invariant prime divisors over $X$.
We compute
$$
\beta(E_{\varphi})>0=\beta(E_{\phi}),
$$
which implies that $X$ is K-semistable by \cite[Corollary 4.14]{Zhuang}.
In particular, we see that $\mathrm{Fut}_X=0$.
On the other hand, $\varrho$ is just the~quotient map $X\dasharrow X/\mathbb{T}$,
where $\mathbb{T}$ is the~maximal torus in $G$.
Now, applying \cite[Proposition 1.38]{Book}, we see that $X$ is K-polystable.

\end{document}